\documentclass[reqno,10pt,a4paper]{amsart}
\usepackage{graphicx,color,amssymb,latexsym,amsfonts,amsmath}
\usepackage{mathrsfs}
\usepackage{graphics}
\usepackage{lscape}
\usepackage{dsfont}
\usepackage[normalem]{ulem} % \sout{}{} for strike fonts
\usepackage{verbatim}

\newtheorem{lemma}{\bf Lemma}[section]

\newtheorem{proposition}{\bf Proposition }[section]%

\newtheorem{theorem}{\bf Theorem}[section]

\numberwithin{equation}{section} \theoremstyle{plain}
\theoremstyle{definition}

\newtheorem{example}{Example}[section]
\newcommand{\E}{\mathsf{E}}
\newcommand{\Q}{\mathsf{Q}}

\newcommand{\why}[1]
    {
    \vskip 0.1in\noindent
        \colorbox{green}
            {
            \begin{minipage}{4.8in}
            {#1}
            \end{minipage}
            }\\ \noindent
    }

\input cyracc.def

\begin{document}

\title[] {The Euler-Maruyama approximations for the CEV model.}

\author{V. Abramov}
\address{School of Mathematical Sciences\\  Monash
University, Building 28\\ Clayton Campus, Wellington road, Victoria
3800\\ Australia} \email{Vyacheslav.Abramov@sci.monash.edu.au;
vabramov126@gmail.ru}
%thanks{}

\author{F. Klebaner}
\address{School of Mathematical Sciences\\  Monash
University, Building 28\\ Clayton Campus, Wellington road, Victoria
3800\\ Australia} \email{fima.klebaner@sci.monash.edu.au}
 \thanks{ Research was supported by the Australian Research
Council Grant DP0881011.}

\author{R. Liptser}
\address{Department of Electrical Engineering Systems,
Tel Aviv University, 69978 Tel Aviv, Israel}
\email{liptser@eng.tau.ac.il; rliptser@gmail.com}


%\commby{}%
% ----------------------------------------------------------------
\maketitle

\begin{abstract} The CEV model is given by the
stochastic differential equation $X_t=X_0+\int_0^t\mu
X_sds+\int_0^t\sigma (X^+_s)^pdW_s$,   $\frac{1}{2}\le p<1$. It
features a non-Lipschitz diffusion coefficient and gets absorbed at
zero with a positive probability. We show the weak convergence of
Euler-Maruyama approximations  $X_t^n$   to the process $X_t$, $0\le
t\le T$, in the Skorokhod metric. We give a new approximation  by
continuous processes which allows to relax some technical conditions
in the proof of weak convergence in \cite{HZa} done in terms of
discrete time martingale problem. We calculate ruin probabilities as
an example of such approximation. We establish that the ruin
probability evaluated by simulations  is not guaranteed to converge
to the theoretical one, because the point zero is a discontinuity
point of the limiting distribution. To establish such convergence we
use the Levy metric, and also confirm the convergence numerically.
Although the result is given for the specific model, our method
works in a more general case of non-Lipschitz diffusion with
absorbtion.
\end{abstract}

\maketitle
% ---------------------------------------------------------
\section{\bf Introduction and the Main Result.}
\label{sec-1}

We consider the Constant Elasticity of Variance (CEV) model (e.g.
\cite{Cox}, \cite{DelShir}) defined by the It\^o equation with
respect to   a Brownian motion $W_t$
\begin{equation}\label{eq:1.0}
X_t=X_0+\int_0^t\mu X_sds+\int_0^t\sigma (X^+_s)^pdW_s,
\end{equation}
where $x^+=0\vee x$,    constants  $p\in\big[\frac{1}{2}, 1\big)$,
$\mu\in \mathbb{R}$, $\sigma>0$, and a positive initial condition
$X_0$. Denote by $\tau$ the first time of hitting zero
  $\tau=\inf\{t:X_t=0\}$.
It is  known that  zero is an absorbing state, e.g. \cite{ShigaWat},
and that $\mathsf{P}(\tau<\infty)>0$.

The aim of this paper is to prove validity of approximations of
expectations of some functionals of these diffusions by the
Euler-Maruyama scheme. For example, in finance this model represents
a price process and it is important to evaluate expectations of
payoffs depending on past prices $\E{\mathcal G}(X_{[0,T]})$; as
well as for evaluation of the ruin probability $\mathsf{P}(\tau\le
T)$ by simulations. In population modeling the  diffusion with
$p=\frac{1}{2}$ represents the size of a population and      is
known as Feller's branching diffusion. In this case $\tau$
represents the time to extinction of the population.

When exact theoretical expressions of such functionals exist then a
comparison of simulated and exact results show how good the
approximation is. However, in many cases exact expressions are not
available. Then one needs to justify such approximations. Weak
convergence established below assures convergence of expectations of
bounded and continuous functionals of simulated values to the exact
ones.

 To allow a larger
family of functionals to be approximated  we show  the weak
convergence of approximations in the Skorokhod metric rather than
uniform. We also consider approximation of the ruin probability
$P(\tau\le t)$, which is the expectation of   the past dependent and
discontinuous in this metric functional  $ I_{\{\tau\le T\}}.$

The fact that the diffusion coefficient is singular and
non-Lipschitz makes the analysis non-standard. Feller in
\cite{Feller} showed for the case   $p=1/2$ that the  solution of
the Fokker-Plank equation exists and is unique and gave its
fundamental solution. This fact is used    for evaluation of ruin
probability by claiming that for any $T$, $X_T$ has a density on
$(0,\infty)$, and in particular its distribution function is
continuous at any positive point.

 Previous various results on the Euler-Maruyama algorithm
 for different type
of diffusion models can be found in Kloeden and Platten
\cite{KlPl} and Milstein and Tretyakov \cite{MilT}. They are
concerned mostly with the approximation of the final value $X_T$,
rather than the whole trajectory $X_t$, $0\le t\le T$. When both
the drift and diffusion coefficients are Lipschitz and the
diffusion coefficient is nonsingular, the standard theory applies.
There is a body of literature on the topic of approximations, see
e.g. Bally and Talay \cite{BaTaL}, Bossy and Diop
\cite{BossyDiop}, Gy\"ongy \cite{G}, Gy\"ongy and Krylov
\cite{GKryl}, Halidias and Kloeden \cite{HalKl}, Higham, Mao and
Stuart \cite{HMS},  Hutzenthaler and Jentzen \cite{MHAJ}. Z\"ahle
\cite{HZa} considers the case when the classical
``Lipschitz-Lipschitz'' setting fails. However, as a rule, the
cases considered in the literature are such that the drift and
diffusion coefficients exclude the  absorbtion effect even when
the diffusion coefficient is singular such as in Bessel diffusion.

We chose to give the results for the particular model of CEV for the
sake of transparency, and because this model is of significance  in
applications. The reader will note that the proofs are sufficiently
involved already in this special case. However, our method and
result hold in a more general case of non-Lipschitz diffusion with
absorbtion.

Recall next the Euler-Maruyama scheme for the diffusion process
$X_t$. Taking  for simplicity
 equidistant partitions of $[0,T]$,
$ 0\equiv t^n_0<t^n_1<\ldots<t^n_{t_{n-1}}<t^n_n\equiv T, \
t^n_k-t^n_{k-1}\equiv\frac{T}{n} $,  it is defined by the
following recursion
\begin{equation}\label{eq:2.1}
\begin{aligned}
X^n_{t^n_0}&=X_0,
\\
X^n_{t^n_k}&= X^n_{t^n_{k-1}}+\mu  \frac{T}{n}
 (X^n_{t^n_{k-1}})^++
\sigma(X^{n+}_{t^n_{k-1}})^p\sqrt{\frac{T}{n}}\xi_k,
\\
X^n_t&=X^n_{t^n_{k-1}}, \quad t\in[t^n_{k-1}, t^n_k),  \quad
k=1,\ldots n,
\end{aligned}
\end{equation}
where $(\xi_k)_{k\ge 1}$ is i.i.d. sequence of $(0,1)$-Gaussian
random variables.

The process $X^n_{t^n_k}$ can be considered as a discrete time
semimartingale (see, e.g., Z\"ahle \cite{HZa}) and $X^n_t$ as a
continuous time semimartingale with discontinuous paths. It may seem
that there is not much difference between  discrete or continuous
time approximations. However, using a continuous time continuous
approximation of $X^n_t$
 introduced in the sequel enables us to simplify
the proof of the weak convergence and avoid some technical
assumptions used in \cite{HZa}, e.g. (2.3) ibid.

 The paths of the  process
$X^n=(X^n_t)_{t\in [0,T]}$ are right continuous piece-wise constant
functions with left limits belonging to the Skorokhod space
$\mathbb{D}=\mathbb{D}_{[0,T]}$. Consider a metric space
$(\mathbb{D}, d_0)$ endowed with the Skorokhod metric $d_0$: if
$\pmb{x}$, $\pmb{y}\in \mathbb{D}$, then
\begin{equation*}
d_0(\pmb{x},\pmb{y})=\inf_{\varphi\in\Phi}\Big\{\sup_{t\in[0,T]}|\pmb{x}_t-\pmb{y}_{\varphi(t)}|
+ \sup_{0\le s<t\le
1}\Big|\log\frac{\varphi(t)-\varphi(s)}{t-s}\Big|\Big\},
\end{equation*}
where $\Phi$ is a set of strictly increasing continuous functions
$\varphi= (\varphi(t))_{0\le t\le T}$ with $\varphi(0)=0$,
$\varphi(T)=1$.

Denote by  $\Q^n$  the probability measure on the  space
$(\mathbb{D},\mathscr{D})$, where $\mathscr{D}$ is the Borel
$\sigma$-algebra, of the distribution of $X^n=(X^n_t)_{t\in[0,T]}$,
and by $\Q$
 the distribution of $
X=(X_t)_{t\in[0,T]}.$ Since $X$ is a continuous process,
$\Q(\mathbb{C})=1$,   where $\mathbb{C}$ denotes the subspace of
continuous functions.

Evaluations of functionals  by simulations is justified by the
convergence of measures in the Skorokhod space
 $ \Q^n\xrightarrow[n\to \infty]{\rm d_0}\Q$, that is,
\begin{equation*}
\lim_{n\to\infty}\int_\mathbb{D}f(\pmb{x})d\Q^n=\int_\mathbb{C}f(\pmb{x})d\Q
\end{equation*}
for any bounded and continuous in the metric $d_0$ function
$f(\pmb{x})$.  The aim of the paper is precisely to show this
convergence. Note that such a  function $f(\pmb{x})$ need not be
continuous in the uniform metric $
\varrho(\pmb{x},\pmb{y})=\sup\limits_{t\in[0,T]}|\pmb{x}_t-\pmb{y}_t|.
$

We remark on the use of the function $x^+$ next. Since the process
$X=(X_t)_{t\in[0,T]}$ is nonnegative, the notation $X^+_t$ in
\eqref{eq:1.0} is used formally. However, it must be used for the
approximating process $X^{n+}_t$ in \eqref{eq:2.1}, because $X^n_t$
can (and does) become negative. Since $X^n_t$ has piece-wise
constant paths  the stopping time
$$
\tau^n=\inf\{t\le T:X^n_t\le 0\}
$$
much more likely corresponds  to the negative value rather than
the zero value of   $X^n_{\tau_n}$. The plus notation $X^{n+}_t$
in \eqref{eq:2.1} enables us to exclude negative values of $X^n_t$
everywhere except of $X^n_{\tau_n}$. Thus, in simulations of
$f(X)$ we must use
 $f(X^{n+})$. Next,
the function $g(\pmb{x}):=f(\pmb{x}^+)$ inherits the properties of
$f(\pmb{x})$ and is bounded and continuous in the metric $d_0$,
moreover, when applied to simulations it converges to the desired
limit
$$
\lim_{n\to\infty}\E f(X^{n+})\equiv\lim_{n\to\infty}\E g(X^n)=\E
g(X)=\E f(X^+)\equiv \E f(X).
$$
We introduce now the continuous approximation $\widetilde{X}^n_t$
used in the proof of weak convergence
$$
\widetilde{X}^n_t=X_0+ \sum_{k=1}^n
\int_{t^n_{k-1}}^{t\wedge\tau^n_k}\mu (\widetilde{X}^n_{t^n_{k-1}})^+ds+
\sum_{k=1}^n \int_{t^n_{k-1}}^{t\wedge t^n_k}\sigma
(\widetilde{X}^{n+}_{t^n_{k-1}})^pd\widetilde{W}_s,
$$
where $\widetilde{W}_t$ is a Brownian motion such that
$\widetilde{W}_{t^n_k}
-\widetilde{W}_{t^n_{k-1}}\equiv\sqrt{\frac{T}{n}}\xi_k$.

Denote $\widetilde{\Q}^n$ the distribution of
$(\widetilde{X}^n_t)_{t\in[0,T]}$. The main result  is that  $
\Q^n\xrightarrow[n\to \infty]{\rm d_0}\Q$.

\begin{theorem}\label{theo-1.1a}

For any $T>0$, the Euler-Maruyama approximation  for the model
\eqref{eq:2.1} converges weakly  in the Skorokhod metric $d_0$ to
the limit process $(X_t)_{t\in [0,T]}$ defined in \eqref{eq:1.0}.
\end{theorem}
The proof is done in three steps: first we show that the distance
in the uniform metric between the Euler-Maruyama   and the
continuous semimartingale approximations converges to zero in
probability, secondly we show that the continuous approximation
converges  in the Skorokhod metric to the solution, and finally we
deduce convergence of the Euler-Maruyama approximation. Thus the
three steps in the proof are

\begin{equation}\label{eq:three}
\begin{aligned}
& \text{\bf step 1.}  \quad
\varrho(X^n,\widetilde{X}^n)\xrightarrow[n\to\infty]{\rm prob.}0
\\
& \text{\bf step 2.}\quad  \widetilde{\Q}^n\xrightarrow[n\to
\infty]{\rm d_0}\Q
\\
& \text{\bf step 3.}\quad \left.
  \begin{array}{ll}
    \varrho(X^n,\widetilde{X}^n)\xrightarrow[n\to\infty]{\rm
prob.}0 &  \\
 \widetilde{\Q}^n\xrightarrow[n\to \infty]{\rm d_0}\Q&
  \end{array}
\right\}\Rightarrow\Q^n\xrightarrow[n\to \infty]{\rm d_0}\Q.
\end{aligned}
\end{equation}
The proof of ``step 1'' uses standard stochastic calculus
calculations for semimartingales. The proof of ``step 2'' follows
from the  results on diffusion approximation for semimartingales
(see \cite[Ch.8]{LSMar}). The proof of ``step 3'' is done as a
standard application of  Billingsley \cite[Theorem 4.4, Ch.1, \S
4]{Bil}.

The paper is organized as follows.   In Section 2 we give
preliminary results used in the proof, with some being of
independent interest. The proof of the Theorem is given in Section
3. Section 4 gives simulations for the ruin probability $
\mathsf{P}(\tau< T). $ To be self contained, the existence and
uniqueness of solutions of
  the stochastic differential equation for the CEV model  \eqref{eq:1.0} is ahown in
Section \ref{sec-4}.

\section{\bf Preliminaries}
\label{sec-2}

We  shall use  the Doob maximal inequality for martingales (see e.g.
\cite[Ch.1, \S 9]{LSMar} and   \cite[p.201]{Klebaner}) in the
following context. Let $(\alpha_k)_{0\le k\le n-1}$ and
$(\xi_k)_{1\le k\le n}$ be random variables such that

- $\E\alpha^2_k<\infty$, $0\le k\le n-1$;

- $(\xi_k)_{1\le k\le n}$ is i.i.d. with $\E\xi_1=0$, $\E\xi^2_1=1$;

- $\xi_k$ and $\{\alpha_0,\ldots,\alpha_{k-1}\}$ are independent for
any $1\le k\le n$.

\smallskip
\noindent Set $M_t=\sum\limits_{j=1}^{[nt]}\alpha_{j-1}\xi_j$, where
$t\in [0,1]$ and $[nt]=k$ if $t\in(\frac{k-1}{n},\frac{k}{n}]$. The
process $M_t$ is a square integrable martingale for a suitable
filtration. So, the Doob maximal inequality $
\E\big(\big|\sup_{t\in[0,1]}M_t\big|^2\big)\le 4 \E M^2_1 $ is
equivalent to

\begin{equation}\label{eq:D1}
\E\max_{1\le k\le n}\Big|\sum_{j=1}^k\alpha_{j-1}\xi_j\Big|^2\le
4\sum_{j=1}^n \E\alpha^2_{j-1}.
\end{equation}

Next we give the   result useful in the analysis of products of
random variables.

\begin{lemma}\label{convProd}
Let $X_n$, $Y_n$ be sequences of random variables such that $X_n$
converges to zero in probability and expectations of $Y_n$ are
uniformly bounded, $\sup_n\E |Y_n|={\bf r}<\infty$. Then $X_nY_n$
converges to zero in probability.
\end{lemma}
\begin{proof}
\begin{gather*}
\mathsf{P}\big(|X_nY_n|>\varepsilon\big) = \mathsf{P}\big(|Y_n|\le
C, |X_nY_n|>\varepsilon\big) +\mathsf{P}\big(|Y_n|> C,
|X_nY_n|>\varepsilon\big)
\\
\le \mathsf{P}\big(C|X_n|>\varepsilon\big) +\mathsf{P}\big(|Y_n|>
C\big).
\end{gather*}
By the Chebyshev inequality $ \mathsf{P}\big(|Y_n|> C\big)\le
\frac{\mathbf{r}}{C}. $ Hence
$$
\varlimsup_{n\to\infty}\mathsf{P}\big(|X_nY_n|>\varepsilon\big) \le
\frac{\mathbf{r}}{C}\xrightarrow[C\to\infty]{}0. $$

\end{proof}

Next we need the following result of convergence to zero the
expectation of double-maximum  the Brownian motion increments over
partitions.

\begin{lemma}\label{maxsup}
Let $W_t$ be a Brownian motion and $\{t_i\}$, $i=1,\dots, n$ an
equidistant partition of $[0,1]$. Let $W^*_i=\sup_{t\in [t_{i-1},
t_i]}|W_t-W_{t_{i-1}}|$, and $M_n=\max_{1\le i\le n}W^*_i$. Then
 $M_n$ converges to zero  in $\mathbb{L}^k$ for any $k\ge 1$.
\end{lemma}
\begin{proof}
We calculate $k$-th moment of $M_n$, $k\ge 1$. Since $M_n\ge 0$
\begin{gather*}
\E  M^k_n=\int_0^\infty \mathsf{P}\big(M^k_n\ge x\big)dx
=\int_0^\infty \mathsf{P}\big(M_n\ge x^{1/k}\big)dx
\\
\le n\int_0^\infty \mathsf{P}\Big( W^*_1\ge x^{1/k}\Big)dx.
\end{gather*}
The inequality is obtained since $(W^*_i)_{1\le i\le n}$ are i.i.d.
random variables, therefore for any $x\ge 0$
$$\begin{small}
 \mathsf{P}\big(M_n\ge x^{1/k}\big)=\mathsf{P}\Big(\max_{1\le i\le
n}W^*_i\ge x^{1/k}\Big)=\mathsf{P}\Big(\bigcup_{i=1}^n \{W^*_i\ge
x^{1/k}\}\Big)\le n\mathsf{P}\Big(W^*_1\ge x^{1/k}\Big) .
\end{small}
$$

Next since $
\sup_{t\in[0,1/n]}|W_t|\le\sup_{t\in[0,1/n]}W_t+\sup_{t\in[0,1/n]}(-W_t),
$ it follows

\begin{gather*}
\mathsf{P}\Big(W^*_1\ge x^{1/k}\Big)\le
\mathsf{P}\bigg(\sup_{t\in[0,1/n]}W_t+\sup_{t\in[0,1/n]}(-W_t)\ge
x^{1/k}\bigg)
\\
\le \mathsf{P}\bigg(\sup_{t\in[0,1/n]}W_t\ge
\frac{1}{2}x^{1/k}\bigg)+\mathsf{P}\bigg(\sup_{t\in[0,1/n]}(-W_t)
\ge \frac{1}{2}x^{1/k}\bigg) =2\mathsf{P}\bigg(W_{1/n}\ge
\frac{1}{2}x^{1/k}\bigg),
\end{gather*}
 where we have used  the well-known law of the maximum of Brownian
 motion.

Thus
$$
\E M^k_n\le 4n\int_0^\infty \mathsf{P}\bigg(W_{1/n}\ge
\frac{1}{2}x^{1/k}\bigg)dx\le 4n\int_0^\infty
\mathsf{P}\bigg(|W_{1/n}|\ge \frac{1}{2}x^{1/k}\bigg)dx
$$
$$
  = 4n2^k\E|W_{1/n}|^k=C_kn^{1-k/2},
$$
where $C_k$ depends only on $k$, ($C_k=2^{k+2}\E |\xi |^k$ with
$\xi\sim N(0,1)$). Hence $M_n$ converges to zero in $\mathbb{L}^k$
for any $k>2$. But since convergence in $\mathbb{L}^p$ for a $p>1$
implies convergence in  $\mathbb{L}^k$ for any $k\in [1,p]$, the
statement is proved.
\end{proof}

The next elementary inequality is new and is  instrumental in the
proof.

\begin{lemma}\label{lem-pdiff}

For any $x,y\in\mathbb{R}$, and $p\in[\frac{1}{2},1)$,

\begin{gather*}
|(x^+)^{2p}-(y^+)^{2p}| \le
  \begin{cases}
|x-y|, & p=\frac{1}{2} \\
 (2+|x|+|y|)|x-y|^p, & p\in\big(\frac{1}{2},1\big).
\end{cases}
\end{gather*}
\end{lemma}

\begin{proof}
For  $p=\frac{1}{2}$ we have

\begin{gather}
|x^+-y^+|=|x-y|I_{\{x>0,y>0\}}+|x|I_{\{x>0,y\le 0\}} +|y|I_{\{x\le
0,y> 0\}}
\nonumber\\
\le |x-y|I_{\{x>0,y>0\}}+|x-y|I_{\{x>0,y\le 0\}} +|y-x|I_{\{x\le
0,y> 0\}}\le |x-y|.
\label{eq:+||}
\end{gather}

\medskip
For $p\in\big(\frac{1}{2}, 1)$  we have
\begin{align*}
|(x^+)^{2p}-(y^+)^{2p}|&=|(x^+)^p-(y^+)^p||(x^+)^p+(y^+)^p|
\\
&\le (2+|x|+|y|)|(x^+)^p-(y^+)^p|.
\end{align*}
Next we now show that
\begin{equation*}
|x^p-y^p|\le|x-y|^p.
\end{equation*}

For $x=y=0$ it is obvious.

Consider $x>0$ and $y<0$. Then, taking into account the proved
inequality $|x^+-y^+|\le |x-y|$ (the statement of this lemma for
$p=\frac{1}{2}$) and the fact that $y^+=0$, we obtain
$$
|(x^+)^p-(y^+)^p|=|x^+|^p=|x^+-y^+|^p\le |x-y|^p.
$$
Clearly, the inequality remains true for $x<0$ and $y>0$.

If both $x$ and $y$ are positive and $x>y$, then
$$
|x^p-y^p|^{1/p}=x\Big|1-\Big(\frac{y}{x}\Big)^p\Big|^{1/p}\le
x\Big|1-\Big(\frac{y}{x}\Big)\Big|^{1/p} \le
x\Big|1-\Big(\frac{y}{x}\Big)\Big|=|x-y|
$$
and, in turn, $ |x^p-y^p|\le|x-y|^p. $ It is easy to see the
inequality remains true for $y>x$.
\end{proof}

\section{\bf Proof of Theorem \ref{theo-1.1a}}
In the following result we show  that the maximum of discrete
approximations have uniformly bounded seconds moments.
 Henceforth  $\mathbf{r}$ denotes a generic positive constant independent
of $n$ with different values at different appearances.

\begin{lemma}\label{lem-000}
Let
 $X^n_{t^n_{k}}$ be the
Euler-Maruyama  approximation defined in \eqref{eq:2.1}.

Then
 $
\E\max\limits_{1\le k\le n}|X^n_{t^n_{k}}|^2\le \mathbf{r}. $
\end{lemma}
\begin{proof}
  First we bound from above  $\max\limits_{1\le k\le
n}\E|X^n_{t^n_k}|^2$. By \eqref{eq:2.1}
\begin{equation}\label{eq:UUU}
\E|X^n_{t^n_k}|^2=\E|X^n_{t^n_{k-1}}|^2\Big(1+\frac{\mu T}{n}\Big)^2
+\sigma^2 \E(X^{n+}_{t^n_{k-1}})^{2p}\frac{T}{n}.
\end{equation}
We use repeatedly the following bound of
$(X^{n+}_{t^n_{k-1}})^{2p}$. Since $(x^+)^{2p}\le |x|^{2p}$ and
$2p<2$, we have $ (x^+)^{2p}\le 1+x^2. $ Hence
\begin{equation}\label{boundsigma}
\E(X^{n+}_{t^n_{k-1}})^{2p}\le 1+\E|X^n_{t^n_{k-1}}|^2.
\end{equation}
    Next, for sufficiently large $n$,
there exists  $\mathbf{r}$ such that
\begin{equation}\label{eq:rr}
 \big(1+\frac{\mu
T}{n}\big)^2\le 1+\frac{\mathbf{r}}{n}.
\end{equation}
Hence for sufficiently large $n$, we obtain from \eqref{eq:UUU} by
using  \eqref{eq:rr} the recurrent inequality
\begin{gather*}
\E|X^n_{t^n_k}|^2\le
\E|X^n_{t^n_{k-1}}|^2\Big(1+\frac{\mathbf{r}}{n}\Big)+\frac{\sigma^2T}{n}
 \Big(1+\E|X^n_{t^n_k}|^2\Big)
 \\
 =\E|X^n_{t^n_{k-1}}|^2\Big(1+\frac{\mathbf{r}+\sigma^2T}{n}\Big)+\frac{\sigma^2T}{n}
 :=\E|X^n_{t^n_{k-1}}|^2\Big(1+\frac{\mathbf{r}}{n}\Big)+\frac{\mathbf{r}}{n}.
\end{gather*}
Iterating it, for   $k\le n$ we obtain
\begin{gather*}
\E|X^n_{t^n_k}|^2\le X^2_0
\Big(1+\frac{\mathbf{r}}{n}\Big)^k+\frac{r}{n}\sum_{j=1}^k\Big(1+\frac{\mathbf{r}}{n}
\Big)^{k-j+1}
\\
\le
(X^2_0+\mathbf{r})\Big(1+\frac{\mathbf{r}}{n}\Big)^n=(X^2_0+\mathbf{r})O\big(e^{\mathbf{r}}\big).
\end{gather*}

Thus
\begin{equation}\label{max2nd}
\max\limits_{1\le k\le n}\E|X^n_{t^n_k}|^2\le \mathbf{r}.
\end{equation}
From the definition of the scheme \eqref{eq:2.1} we obtain by
iterations

\begin{equation*}
\max_{1\le k\le n}|X^n_{t^n_k}|\le
|X_0|+|\mu|\frac{T}{n}\sum_{j=1}^n |X^n_{t^n_{j-1}}|+
\sqrt{\frac{T}{n}}\max_{1\le k\le
n}\bigg|\sum_{j=1}^k\sigma(X^{n+}_{t^n_{j-1}})^p\xi_j\bigg|.
\end{equation*}
Next we use the Cauchy-Schwarz inequality $(\sum_{i=1}^la_i)^2\le l
\sum_{i=1}^la_i^2$ with $l=3$. Proceeding from above we have
\begin{multline}\label{itermax}
\E \max_{1\le k\le n}|X^n_{t^n_k}|^2
\\
\le 3 |X_0|^2+3\frac{\mu^2T^2}{n^2} \E\Big(\sum_{j=1}^n
|X^n_{t^n_{j-1}}| \Big)^2  +3\frac{T}{n}  \E\max_{1\le k\le
n}\Big|\sum_{j=1}^k\sigma(X^{n+}_{t^n_{j-1}})^p\xi_j\Big|^2.
\end{multline}
The Cauchy-Schwarz inequality  $\Big(\sum_{j=1}^n |X^n_{t^n_{j-1}}|
\Big)^2\le n  \sum_{j=1}^n (X^n_{t^n_{j-1}})
 ^2$  coupled with the bound on the
largest second moment shown above \eqref{max2nd} implies that the
second term in \eqref{itermax} is bounded by a constant $\mathbf{r}$
independent of $n$. The last term is bounded by using the Doob's
maximal inequality \eqref{eq:D1} and the bound in \eqref{boundsigma}
\begin{gather*}
\E\max_{1\le k\le n}\Big|\sum_{j=1}^k\sigma(X^{n+}_{t^n_{j-1}})^p
\xi_j\Big|^2 \le 4\sigma^2\sum_{j=1}^n\E(X^n_{t^n_{j-1}})^{2p}
\le4\sigma^2\sum_{j=1}^n[1+ \E(X^n_{t^n_{j-1}})^{2}].
\end{gather*}
Using \eqref{max2nd} we can now see that the last term in
\eqref{itermax} is bounded by $\mathbf{r}$.
 \end{proof}

We proceed now to prove the
 steps in \eqref{eq:three}.
\subsection{Step 1} The distance between the two approximations in the
$\sup$ norm converges to zero in probability and in $\mathbb{L}^2$.
\begin{lemma}\label{lem-2.3b}
$\E \varrho^2 (X^n,\widetilde{X}^n)\xrightarrow[n\to\infty]{}0. $
\end{lemma}
\begin{proof}
From  the definitions of $X^n_t$ and $\widetilde{X}^n_t$ it follows
that at the points of the partitions both approximations coincide
and at intermediate points the following holds
\begin{align}\label{eq:==}
X^n_{t^n_k}&\equiv \widetilde{X}^n_{t^n_k}
\nonumber\\
\widetilde{X}^n_t-X^n_{t^n_{k-1}}&=\int_{t^n_{k-1}}^{t\wedge
t^n_k}\mu (X^n_{t^n_{k-1}})^+ds +\int_{t^n_{k-1}}^{t\wedge
t^n_k}\sigma(X^{n+}_{t^n_{k-1}})^pd\widetilde{W}_s, \quad
t\in[t^n_{k-1}, t^n_k].
\end{align}

By the formula \eqref{eq:==},
\begin{gather*}
\sup_{t\in[t^n_{k-1}, t^n_k]}|\widetilde{X}^n_t-X^n_{t^n_{k-1}}|\le
\frac{|\mu|T}{n}|X^n_{t^n_{k-1}}| +\sigma\sup_{t\in[t^n_{k-1},
t^n_k]}\big|\widetilde{W}_t-\widetilde{W}_{t^n_{k-1}}|
|X^n_{t^n_{k-1}}|^p.
\end{gather*}
Consequently,
\begin{gather*}\label{supdif}
\varrho
(X^n,\widetilde{X}^n)=\sup_{t\in[0,T]}|\widetilde{X}^n_t-X^n_t|=
\max\limits_{1\le k\le n}\sup_{t\in[t^n_{k-1},
t^n_k]}|\widetilde{X}^n_t-X^n_{t^n_{k-1}}|\nonumber
\\
\le \frac{|\mu|T}{n}\max_{1\le k\le n}|X^n_{t^n_{k-1}}| +\sigma
\max\limits_{1\le k\le n}\sup_{t\in[t^n_{k-1},
t^n_k]}\big|\widetilde{W}_t-\widetilde{W}_{t^n_{k-1}} | \max_{1\le
k\le n} |X^n_{t^n_{k-1}}|^p.
\end{gather*}
The first term converges to zero in $\mathbb{L}^2$
 using Lemma \ref{lem-000}   as  $
 \E   \max_{1\le k\le n}|X^n_{t^n_{k-1}}|\le {\bf r}.$
For the second term use H\"older's inequality with parameters
$\frac{2}{p}$ and $\frac{2}{2-p}$:
\begin{gather*}
\E \bigg(\max\limits_{1\le k\le n}\sup_{t\in[t^n_{k-1},t^n_k]}
\big|\widetilde{W}_t-\widetilde{W}_{t^n_{k-1}} | \max_{1\le k\le n}
|X^n_{t^n_{k-1}}|^p\bigg)^2
\\
\le \Bigg(\E\bigg[\max\limits_{1\le k\le n}\sup_{t\in[t^n_{k-1},
t^n_k]}\big|\widetilde{W}_t-\widetilde{W}_{t^n_{k-1}}
|\bigg]^{\frac{2}{2-p}}\Bigg)^{2-p}\Bigg(\E \max_{1\le k\le n}
|X^n_{t^n_{k-1}}|^{2}\Bigg)^{p}.
\end{gather*}
By Lemma
 \ref{lem-000} $\sup_n\E\max_{1\le k\le n}
|X^n_{t^n_{k-1}}|^2\le \mathbf{r}$. Since $\frac{2}{2-p}>1$, by
Lemma \ref{maxsup}
$$
\lim_{n\to\infty}\E\Big(\max\limits_{1\le k\le
n}\sup_{t\in[t^n_{k-1},
t^n_k]}\big|\widetilde{W}_t-\widetilde{W}_{t^n_{k-1}}
|\Big)^{2/(2-p)}=0.
$$
\end{proof}

\subsection{Step 2} Weak convergence of continuous approximations.

\begin{lemma}\label{lem-4.2}
$ \widetilde{\Q}^n\xrightarrow[n\to \infty]{\rm d_0}\Q. $
\end{lemma}
\begin{proof} The proof rests on a general result on the weak convergence of semimartingales
to a diffusion \cite{LSMar}, Theorem 1, Ch. 8, \S 3. This theorem
states that for weak convergence to a diffusion it is enough to
check convergence of the drifts and quadratic variations evaluated
at the pre-limit processes. The processes $X$ and
$\widetilde{X}^n$ are semimartingales with following
decompositions
\begin{align*}
X_t&=X_0+\underbrace{\int_0^t\mu
X_sds}_{:=B_t(X)}+\underbrace{\int_0^t\sigma (X^+_s)^pdW_s}
_{:=M_t(X)}
\\
\widetilde{X}^n_t&=X_0+ \underbrace{\sum_{k=1}^n
\int_{t^n_{k-1}}^{t\wedge t^n_k}\mu
\widetilde{X}^{n+}_{t^n_{k-1}}ds}_{:=B^n_t(\widetilde{X}^n)}+
\underbrace{\sum_{k=1}^n \int_{t^n_{k-1}}^{t\wedge t^n_k}\sigma
(\widetilde{X}^{n+}_{t^n_{k-1}})^pd\widetilde{W}_s}_
{:=M^n_t(\widetilde{X}^n)},
\end{align*}
where we have denoted above
\begin{itemize}
  \item[(B)] $B_t(X)$ and $B^n_t(\widetilde{X}^n)$ are drifts
  \item[(M)] $M_t(X)$ and $M^n_t(\widetilde{X}^n)$ are continuous martingales
  with predictable quadratic variations
\begin{equation*}
\begin{aligned}
\langle M\rangle_t(X)&=\int_0^t\sigma^2(X^+_s)^{2p}ds,
\\
\langle M^n\rangle_t(\widetilde{X}^n)&=\sum\limits_{k=1}^n
\int_{t^n_{k-1}}^{t\wedge t^n_k}\sigma^2(\widetilde{X}^{n+}_
{t^n_{k-1}})^{2p}ds.
\end{aligned}
\end{equation*}
\end{itemize}

\smallskip
\noindent  The above mentioned Theorem 1 of \cite{LSMar}, adapted
to the present setting,  states that the weak convergence takes
place if the following three conditions hold.

\begin{equation*}
\left.
  \begin{array}{ll}
{\bf (a)}\quad    \text{Equation (\ref{eq:1.0}) has a unique (at least weak) solution} & \\
{\bf (b)}\quad
\sup\limits_{t\in[0,T]}|B_t(\widetilde{X}^n)-B^n_t(\widetilde{X}^n)|\xrightarrow[n\to\infty]
{\rm prob.}0  & \\
{\bf (c)} \quad \sup\limits_{t\in[0,T]}|\langle
M\rangle(\widetilde{X}^n)-\langle M^n\rangle_t(\widetilde{X}^n)|
\xrightarrow[n\to\infty]{\rm prob.}0    , &
  \end{array}
\right\}\Rightarrow\widetilde{\Q}^n\xrightarrow[n\to\infty]{\rm
d_0}\Q.
\end{equation*}
We proceed to verify these conditions. The existence and uniqueness
of (\ref{eq:1.0})   is known, e.g. \cite{ShigaWat}, and is also
given for completeness   Proposition \ref{pro-.A}. Hence {\bf (a)}
holds. To show  {\bf (b)} write, taking into account
\eqref{eq:+||},
\begin{align*}
&\sup_{t\in[0,T]}\Big|B_t(\widetilde{X}^n)-B^n_t(\widetilde{X}^n)]\Big|
\le
|\mu|\sum_{k=1}^n\int_{t^n_{k-1}}^{t^n_k}|\widetilde{X}^n_s-\widetilde{X}^n_{t^n_{k-1}}|ds
\le T|\mu|\varrho(\widetilde{X}^n, X^n).
\end{align*}
 Hence  {\bf (b)} holds by applying Lemma \ref{lem-2.3b}.

To prove {\bf (c)}  write the bound
\begin{gather}
\sup_{t\in[0,T]}\Big|\langle M\rangle_t(\widetilde{X}^n)- \langle
M^n\rangle_t(\widetilde{X}^n)\Big|\le\sigma^2
\sum_{k=1}^n\int_{t^n_{k-1}}^{t^n_k}|(\widetilde{X}^{n+}_s)^{2p}-(\widetilde{X}^{n+}_{t^n_{k-1}})
^{2p}|ds\nonumber
\\
=\sigma^2
\sum_{k=1}^n\int_{t^n_{k-1}}^{t^n_k}|(\widetilde{X}^{n+}_s)^{2p}-(X^{n+}_{t^n_{k-1}})
^{2p}|ds,\label{qvbound}
\end{gather}
where we have used that the two approximations coincide on the grid.
By applying Lemma \ref{lem-pdiff} we have further bound on the
expression under the integral

\begin{multline*}
|\widetilde{X}^{n+}_s)^{2p}-(X^{n+}_{t^n_{k-1}})^{2p}|
\\
\le
  \begin{cases}
 |\widetilde{X}^n_s-X^n_{t^n_{k-1}}|, & p=\frac{1}{2}  \\
\Big(2+
\sup\limits_{t\in[0,T]}|\widetilde{X}^n_t|+\sup\limits_{t\in[0,T]}|X^n_t|
 \Big)\sup\limits_{s\in[t^n_{k-1},
t^n_k]}|\widetilde{X}^n_s-X^n_{t^n_{k-1}}|^p, &
p\in\big(\frac{1}{2}, 1\big)
  \end{cases}.
\end{multline*}

Hence for $p=\frac{1}{2}$ the bound in \eqref{qvbound} becomes
$$
\sigma^2
\sum_{k=1}^n\int_{t^n_{k-1}}^{t^n_k}|(\widetilde{X}^{n+}_s)^{2p}-(X^{n+}_{t^n_{k-1}})
^{2p}|ds\le\sigma^2T\varrho(\widetilde{X}^n,X^n)
$$
and the statement follows by Lemma \ref{lem-2.3b}.

For $p\in(\frac{1}{2},1)$ the bound in \eqref{qvbound} becomes
\begin{align*}
&\sigma^2
\sum_{k=1}^n\int_{t^n_{k-1}}^{t^n_k}|(\widetilde{X}^{n+}_s)^{2p}-(X^{n+}_{t^n_{k-1}})
^{2p}|ds=\sigma^2
\sum_{k=1}^n\int_{t^n_{k-1}}^{t^n_k}|(\widetilde{X}^{n+}_s)^{2p}-(X^{n+}_{t^n_{k-1}})
^{2p}|ds
\\
&\le\sigma^2 \Big(2+
\sup\limits_{t\in[0,T]}|\widetilde{X}^n_t|+\sup\limits_{t\in[0,T]}|X^n_t|
 \Big)\sum_{k=1}^n\int_{t^n_{k-1}}^{t^n_k}\sup\limits_{s\in[t^n_{k-1},
t^n_k]}|\widetilde{X}^n_s-X^n_s|^pds
\\
&\le T\sigma^2 \Big(2+
\sup\limits_{t\in[0,T]}|\widetilde{X}^n_t|+\sup\limits_{t\in[0,T]}|X^n_t|\Big)
 \varrho^p(\widetilde{X}^n,X^n).
\end{align*}

By Lemma \ref{lem-000}, Lemma \ref{lem-2.3b} we have the product of
two terms, one of which has uniformly bounded expectations and the
second converges in probability to zero. By Lemma \ref{convProd} the
product converges in probability to zero.

Thus the conditions of the Theorem 1 of \cite{LSMar} are verified
and   weak convergence is proved.
\end{proof}

\subsection{Step 3} Weak convergence of the Euler-Maruyama
approximations.

\begin{lemma}\label{lem-5.q}
For any bounded and continuous in the metric $d_0$ function
$f(\pmb{x})$,
$$
\varlimsup_{n\to\infty}|\E f(X^n)-\E f(X)|
\le\varlimsup_{n\to\infty}|\E f(\widetilde{X}^n)-\E f(X)|=0.
$$
\end{lemma}
\begin{proof}
The result follows from the triangular inequality
\begin{equation}\label{convEM}
|\E f(X^n)-\E f(X)|\le  \E|f(X^n)-f(\widetilde{X}^n)|+|\E
f(\widetilde{X}^n)-\E f(X)|.
\end{equation}
Since convergence in the uniform metric implies convergence in the
Skorokhod metric,   by Lemma \ref{lem-2.3b} we have, since $f$ is
continuous in this metric,
 $$
 \lim_{n\to\infty}\E|f(X^n)-f(\widetilde{X}^n)|=0.
 $$ Taking now the
 $\limsup$ in \eqref{convEM} and using the previous result of
 weak convergence of $\widetilde{X}^n$ to $X$ proves the step 3.
\end{proof}

\section{\bf Evaluation of ruin probability on a finite time interval
by simulations.}
\label{sec-5}

In this section   we evaluate numerically a ruin probability
$\mathsf{P}(\tau\le T)$ on a time interval $[0,T]$, where
$\tau=\{t:X_t=0\}$ by     Euler-Maruyama approximations. An
explicit formula to the ruin probability, as a function of
arguments $p, X_0, \mu, \sigma, T$,  is known in the case
$p=\frac{1}{2},   T=\infty$:
\begin{equation*}
\mathsf{P}(\tau<\infty)=\exp\Big(-\frac{2\mu}{\sigma^2}X_0\Big)
\end{equation*}
(see, e.g. \cite[p. 354]{Klebaner}), but not for finite $T$.

Naturally,     $\mathsf{P}(\tau\le T)$ is important in
applications, and we study it for different values of parameters
$p, X_0, \mu, \sigma, T$:
$$
p=
  \begin{cases}
    \frac{1}{2} &
     \\
    \frac{3}{4},&
  \end{cases}
  X_0=
  \begin{cases}
    \frac{1}{10} &
     \\
    \frac{1}{4}&
    \\
    1,
  \end{cases}
  \mu=
  \begin{cases}
    1 &
    \\
    -1, &
  \end{cases}
  \sigma=1,
T=\begin{cases}
    3 &
     \\
    9.&
  \end{cases}
$$

We combine the Euler-Maruyama simulation algorithm and the
Monte-Carlo technique with $10^{3}$ runs per point.

The basis for analysis is an obvious  formula $ \mathsf{P}(\tau\le
T)=\mathsf{P}(X_T=0) $. It allows us to deal with the distribution
function $F(x):=\mathsf{P}(X_T\le x)$ of $X_T$ instead of harder
to compute distribution function of $\tau$, $ \mathsf{P}(\tau\le
T). $ Notice that
$$
F(x)=
  \begin{cases}
    0 , & x<0 \\
    F(0)=\mathsf{P}(X_T=0)>0 , & x=0,
  \end{cases}
$$
that is, $F(0)-F(0-)>0$ and so the distribution function $F(x)$
has an atom at the point $\{0\}$.

The measure $\Q$ is supported on the space of continuous
functions. So the weak convergence of processes and measures  $
\Q^n\xrightarrow[n\to\infty] {\rm d_0}\Q $ implies weak
convergence of finite dimensional distributions, and in particular
marginals, $ X^n_T\xrightarrow[n\to\infty]{\rm law} X_T $. That is
if
 $ F_n(x)=\mathsf{P}(X^n_T\le x)  $ then
$\lim_{n\to \infty}F_n(x)=F(x)$ at any point of continuity of $F$.
Unfortunately    $0$, our point of interest, is an atom of $F$ and
we can not claim that
\begin{equation*}
\lim_{n\to \infty}F_n(0)=F(0)=\mathsf{P}(\tau\le T).
\end{equation*}

In view of this uncertainty, we give   approximations for lower
and upper bounds of $F(0)$ by using   the L\'evy metric (see e.g.
\cite{HT}):
$$
\mathcal{L}(F_n,F)=\inf\{h>0:F_n(x-h)-h\le F(x)\le F_n(x+h)+h; \
\forall \ x\}.
$$
It is known that weak convergence of distributions implies
convergence in the L\'evy metric $  \Q^n\xrightarrow[n\to\infty]
{\rm d_0}\Q \Rightarrow
 \lim\limits_{n\to\infty}\mathcal{L}(F_n,F)=0.$
\why{Though $\lim\limits_{n\to\infty}\mathcal{L}(F_n,F)=0$ does not
fish out atom it helps localize it. Namely we take}
$x=0$ and
a small suitable $\varepsilon_n$ , in each case determined
experimentally,  such that the interval
$[F_n(-\varepsilon_n)-\varepsilon_n,
F_n(\varepsilon_n)+\varepsilon_n]$ is  small enough and declare that
an estimate of $F(0)$ belongs to this interval. Such values of
$\varepsilon_n$ are pointed out below as last values in each Table.

Simulations for various values of parameters show   good fit of this
procedure.
\begin{small}
\begin{example}\label{ex:ex.4.1}
$p=\frac{1}{2}$, $\mu=-1$, $\sigma=1$, $X_0=\frac{1}{4}$, $T=3$
\begin{table}[ht]  \centering
 \begin{tabular}{|c|c|c|c|c|c|}
   \hline
       \cline{1-4} $\varepsilon$  & $\mathsf{P}(X^n_T\le -\varepsilon)-\varepsilon$
       &$\mathsf{P}(X^n_T\le \varepsilon)+\varepsilon$&
       $\mathsf{P}(X_T=0)$\\  \hline
$3\cdot10^{-6}$                & .9679  & .9738  & [.9679,  .9738]\\
$2\cdot10^{-6}$                & .9700  & .9738  & [.9700,  .9738] \\
\fbox{$10^{-6}=\varepsilon_n$} & .9738  & .9738  & [.9738,  .9738] \\
 \hline
 \end{tabular}
\caption{} \label{Tab1}
 \end{table}

\end{example}

\newpage
\begin{example}\label{ex:ex.4.2}
$p=\frac{1}{2}$, $\mu=+1$, $\sigma=1$, $X_0=\frac{1}{10}$, $T=9$
\begin{table}[ht]  \centering
 \begin{tabular}{|c|c|c|c|c|c|}
   \hline
       \cline{1-4} $\varepsilon$  & $\mathsf{P}(X^n_T\le -\varepsilon)-\varepsilon$
       &$\mathsf{P}(X^n_T\le \varepsilon)+\varepsilon$&
       $\mathsf{P}(X_T=0)$\\  \hline
$2\cdot10^{-6}$                      & .8166  & .8192  & [.8166, .8192] \\
$10^{-6}$                            & .8174  & .8192  & [.8174, .8192] \\
\fbox{$5\cdot10^{-7}=\varepsilon_n$} & .8182  & .8192  & [.8182, .8192]\\
 \hline
 \end{tabular}
\caption{} \label{Tab2}
 \end{table}
\end{example}
\begin{example}\label{ex:ex.4.3}
$p=\frac{1}{2}$, $\mu=+1$, $\sigma=1$, $X_0=\frac{1}{4}$, $T=9$
\begin{table}[ht]  \centering
 \begin{tabular}{|c|c|c|c|c|c|}
   \hline
       \cline{1-4} $\varepsilon$  & $\mathsf{P}(X^n_T\le -\varepsilon)-\varepsilon$
       &$\mathsf{P}(X^n_T\le \varepsilon)+\varepsilon$&
       $\mathsf{P}(X_T=0)$\\  \hline
$2\cdot10^{-6}$                      & .5960 & .5970  & [.5960, .5970] \\
$10^{-6}$                            & .5964 & .5970  & [.5964, .5970] \\
\fbox{$5\cdot10^{-7}=\varepsilon_n$} & .5970 & .5970  & [.5970, .5970]\\
 \hline
 \end{tabular}
\caption{} \label{Tab3}
 \end{table}
\end{example}

\begin{example}\label{ex:ex.4.4}
$p=\frac{1}{2}$, $\mu=+1$, $\sigma=1$, $X_0=1$, $T=9$
\begin{table}[ht]  \centering
 \begin{tabular}{|c|c|c|c|c|c|}
   \hline
       \cline{1-4} $\varepsilon$  & $\mathsf{P}(X^n_T\le -\varepsilon)-\varepsilon$
       &$\mathsf{P}(X^n_T\le \varepsilon)+\varepsilon$&
       $\mathsf{P}(X_T=0)$\\  \hline
$3\cdot10^{-6}$                & .1344  & .1348  & [.1344, .1348]\\
$2\cdot10^{-6}$                & .1346  & .1348  & [.1346, .1348] \\
\fbox{$10^{-6}=\varepsilon_n$} & .1346  & .1348  & [.1346, .1348] \\
 \hline
 \end{tabular}
\caption{} \label{Tab4}
 \end{table}
\end{example}

\begin{example}\label{ex:ex.4.5}
$p=\frac{3}{4}$, $\mu=+1$, $\sigma=1$, $X_0=\frac{1}{10}$, $T=9$
\begin{table}[ht]  \centering
 \begin{tabular}{|c|c|c|c|c|c|}
   \hline
       \cline{1-4} $\varepsilon$  & $\mathsf{P}(X^n_T\le -\varepsilon)-\varepsilon$
       &$\mathsf{P}(X^n_T\le \varepsilon)+\varepsilon$&
       $\mathsf{P}(X_T=0)$\\  \hline
$3\cdot10^{-8}$                        & .6040  & .6206  & [.6040, .6206]\\
$5\cdot10^{-9}$                        & .6040  & .6126  & [.6126, .6206] \\
\fbox{$2.5\cdot10^{-9}=\varepsilon_n$} & .6180  & .6206  & [.6206, .6206] \\
 \hline
 \end{tabular}
\caption{} \label{Tab5}
 \end{table}
\end{example}

\newpage
\begin{example}\label{ex:ex.4.6}
$p=\frac{3}{4}$, $\mu=+1$, $\sigma=1$, $X_0=\frac{1}{4}$, $T=9$
\begin{table}[ht]  \centering
 \begin{tabular}{|c|c|c|c|c|c|}
   \hline
       \cline{1-4} $\varepsilon$  & $\mathsf{P}(X^n_T\le -\varepsilon)-\varepsilon$
       &$\mathsf{P}(X^n_T\le \varepsilon)+\varepsilon$&
       $\mathsf{P}(X_T=0)$\\  \hline
$3\cdot10^{-8}$                        & .3794  & .3864   & [.3794, .3864]\\
$5\cdot10^{-9}$                        & .3816  & .3864   & [.3816, .3864] \\
\fbox{$2.5\cdot10^{-9}=\varepsilon_n$} & .3838  & .3864   & [.3838, .3864] \\
 \hline
 \end{tabular}
\caption{} \label{Tab.6}
 \end{table}
\end{example}

\begin{example}\label{ex:ex.4.7}
$p=\frac{3}{4}$, $\mu=+1$, $\sigma=1$, $X_0=1$, $T=9$
\begin{table}[ht]  \centering
 \begin{tabular}{|c|c|c|c|c|c|}
   \hline
       \cline{1-4} $\varepsilon$  & $\mathsf{P}(X^n_T\le -\varepsilon)-\varepsilon$
       &$\mathsf{P}(X^n_T\le \varepsilon)+\varepsilon$&
       $\mathsf{P}(X_T=0)$\\  \hline
$2\cdot10^{-8}$                        & .0752     & .0790    & [.0752, .0790]\\
$10^{-8}$                              & .0768     & .0790    & [.0768, .0790] \\
\fbox{$2.5\cdot10^{-9}=\varepsilon_n$} & .0782     & .0790    & [.0782, .0790] \\
 \hline
 \end{tabular}
\caption{} \label{Tab.7}
 \end{table}
\end{example}

\begin{example}\label{ex:ex.4.8}
$p=\frac{3}{4}$, $\mu=-1$, $\sigma=1$, $X_0=\frac{1}{3}$, $T=3$
\begin{table}[ht]  \centering
 \begin{tabular}{|c|c|c|c|c|c|}
   \hline
       \cline{1-4} $\varepsilon$  & $\mathsf{P}(X^n_T\le -\varepsilon)-\varepsilon$
       &$\mathsf{P}(X^n_T\le \varepsilon)+\varepsilon$&
       $\mathsf{P}(X_T=0)$\\  \hline
$10^{-8}$                                & .8536  & .8803  & [.8536, .0790]\\
$5\cdot 10^{-9}$                         & .8700  & .8803  & [.8700,  .0790] \\
\fbox{$2.5\cdot10^{-9}=\varepsilon_n$}   & .8757  & .8803  & [.8757, .0790] \\
 \hline
 \end{tabular}
\caption{} \label{Tab.8}
 \end{table}
\end{example}

\end{small}

\section{\bf  Existence and Uniqueness of solution in the CEV model}
\label{sec-4}
\begin{proposition}\label{pro-.A}
The It\^o's equation \eqref{eq:1.0} possesses a unique strong
nonnegative solution.
\end{proposition}
\begin{proof}
 Define a sequence of processes indexed by integers $i$, $i> 1/X_0$, $(X_t^i)_{t\ge 0}$, such
 that $X_t^i$ is a strong solutions to the following sde
\begin{equation}\label{eq:i}
dX^i_t=\mu X^i_tdt+\sigma(i^{-1}\vee |X^i_t|)^pdW_t,\;\; X^i_0=X_0.
\end{equation}
The diffusion coefficient  $ \sigma(i^{-1}\vee |x|)^p $ is Lipschitz
continuous therefore $X_t^i$ is the unique strong solution of
\eqref{eq:i}. Set $\vartheta_i=\inf\{t:X^i_t= i^{-1}\}$. Note that
for $t\le \vartheta_i$
$$
X_t^{i+1}=X_t^i,
$$
and it follows that  $\vartheta_{i+1}>\vartheta_i$. A strong
solution to \eqref{eq:1.0} is constructed by a natural prolongation
$$
 X_{t\wedge\tau}:=\sum\limits_{i\ge
n}X^i_{\vartheta_i}I_{\{\vartheta_i\le t<\vartheta_{i+1}\}}, \
\tau=\lim_{i\to\infty}\vartheta_i.
$$

Finally, Yamada-Watanabe's theorem (see, e.g.  \cite[Rogers and
Williams, p. 265]{RW}) guaranties the   uniqueness of the strong
solution of equation \eqref{eq:1.0}, because  the H\"older parameter
$p\ge \frac{1}{2}$ .
\end{proof}

\end{document}